\documentclass[12pt, reqno]{ijnaa}
\usepackage{amsmath, hyperref, graphicx}
\usepackage[latin9]{inputenc}
\usepackage{array}
\usepackage{url}
\usepackage{multirow}
\usepackage{amsmath}
\usepackage{esint}
\usepackage[numbers]{natbib}

\begin{document}

\setcounter{page}{1}

\title[A numerical scheme for space-time fractional
advection-dispersion equation] {A numerical scheme for space-time fractional
	advection-dispersion equation}

\author[a1]{Shahnam Javadi\corref{c}}
\email{javadi@khu.ac.ir}

\author[a1]{Mostafa Jani}
\email{mostafa.jani@gmail.com}

\author[a1]{Esmail Babolian}
\email{babolian@khu.ac.ir}

\address[a1]{Department of Mathematics,
Faculty of Mathematical Sciences and Computer, Kharazmi University,
Tehran, Iran}

\cortext[c]{Corresponding author}

\vol{x~~(xxxx)~~No. x}

\pages{~xx-xx}



\authors{Javadi, Jani, Babolian}

\begin{abstract}
In this paper, we develop a numerical resolution of the space-time
fractional advection-dispersion equation. The main idea of the present
method is that we utilize spectral-collocation method combining with
a product integration technique in order to discretize the terms involving
spatial fractional order derivatives and it leads to a simple evaluation
of the related terms. By using the Bernstein polynomial approximation,
the problem is transformed into a linear system of algebraic equations.
The error analysis and the order of convergence of the proposed algorithm
are also discussed. Some numerical experiments are presented to demonstrate
the effectiveness of the proposed method and to confirm the analytic
results.


\end{abstract}

\maketitle


\section{{ Introduction}}

\noindent In recent decades, many physical processes have been modeled in terms
of fractional partial differential equations (FPDEs). This kind of
mathematical modeling leads to better agreement with data obtained
in lab experiments than the classical models involving integer order
derivatives. For example, among recently developed models, Suzuki
et al. \citep{Suzu} proposed a fractional advection-dispersion equation
(FADE) for description of mass transport in a fractured reservoir.
They also discussed a FADE model for the evaluation of the effects
of cold-water injection into an advection-dominated geothermal reservoir
in fault-related structures in geothermal areas \citep{Suzu2}. Also,
the particle\textquoteright s motion in crowded cellular environments
which represent anomalous dispersion is formulated as a FPDE \citep{Soko}.
Uchaikin and Sibatov used a FPDE model for theoretical description
of charge carrier transport in disordered semiconductors \citep{Ucha}.

Recently, some numerical methods have been developed for the fractional
advection-dispersion equations, most of them for time fractional advection-dispersion
equations \citep{Gao,Rame,Shirzadi,Stok,Jiang} and some for space
fractional advection dispersion equations \citep{Pang,Sous,Wang}.
However, there are some physical problems that are modeled with both
space and time fractional advection-dispersion equation like space-time
fractional Fokker-Planck equation which is an effective tool for the
processes with both traps and flights, in which the time fractional
term characterizes the traps and the space fractional term characterizes
the flights \citep{Deng}. In spite of this and due to the fact that
the numerical methods for these problems are very challenging, there
are a few numerical schemes developed for these equations.

To the best of our knowledge, no numerical scheme based on product
integration method and Bernstein collocation has been developed for
the space-time fractional advection-dispersion equation. The main
contribution of the current work is to implement a product integration
technique for terms involving space fractional derivatives in the
fractional advection-dispersion problem and to use Bernstein polynomials
with the collocation method in order to transform the problem into
an algebraic linear system. 

In comparison with local approximation methods such as the finite
difference and finite element method, the proposed method leads to
a smaller size of coefficient matrix and less computational effort
that is required for a specific accuracy.

This paper is concerned with providing a numerical scheme for the
Caputo space-time fractional advection dispersion equation, given
by
\begin{eqnarray}
{D_{t}^{\alpha}}u\left(x,t\right)=\kappa_{1}{D_{x}^{\beta}}u\left(x,t\right)-\kappa_{2}D_{x}^{\gamma}u\left(x,t\right),\quad\left(x,t\right)\in\Omega=\left(0,L\right)\times\left(0,\infty\right),\label{eq:main}
\end{eqnarray}
where $x$ is the spatial coordinate and $t$ represents time, $u$
is the concentration, $\alpha\in\left(0,1\right)$ is the temporal
order, $\beta\in\left(1,2\right),\,\gamma\in\left(0,1\right)$ are
the orders of spatial fractional derivatives. Moreover, the operators
$D_{t}^{\alpha},$ $D_{x}^{\beta}$ and $D_{x}^{\gamma}$ stand for
time and space fractional derivatives in the sense of Caputo definition
as described in Definition \ref{CaputoDef}. The positive constants
$\kappa_{1}$ and $\kappa_{2}$ denote the anomalous dispersion and
advection coefficients, respectively. When $\alpha=\gamma=1$ and
$\beta=2$, the classical advection dispersion equation is obtained.
In \citep{Deng2} it is shown that the order of space fractional derivative
regarding anomalous dispersion is mostly occurred between 1.4 and
2 to best fit the experimental data and in \citep{Huan} it varies
between 1.7 and 1.8. 

Some authors discussed special cases of problem (\ref{eq:main}),
for instance Gao and Sun \citep{Gao} proposed a numerical method
based on finite difference method for time fractional case, i.e.,
$\beta=2,\gamma=1$ and Deng \citep{Deng} discussed finite element
method in the case $\gamma=1$ with Riemann-Liouville space fractional
derivatives. 

We consider equation (\ref{eq:main}) subject to the following initial
and homogeneous boundary conditions 
\begin{eqnarray}
 &  & u\left(x,0\right)=g\left(x\right),\quad x\in\left[0,L\right],\label{IV}\\
 &  & u\left(0,t\right)=0,\quad u\left(L,t\right)=0,\quad t>0.\label{BVs}
\end{eqnarray}

The paper is organized as follows. In section \ref{sec:Pre}, we briefly
introduce some background material on fractional derivatives. We also
provide the basic idea of product integration rules. Section \ref{sec:Bern}
is devoted to properties of Bernstein polynomials. In this section
we derive two new representation for first and second order derivatives
of Bernstein polynomials. Time discretization of the given problem
is discussed in section \ref{sec:Time-discret}. In section \ref{sec:Main},
we use time discretization results discussed in previous section and
we utilize Bernstein polynomials with product integration in order
to transfer the given problem into solving a simple algebraic system.
Matrix formulation, error analysis and convergence of the method are
discussed in this section. Some numerical experiments are presented
in section \ref{sec:Ex} to show the efficiency and numerical accuracy
of the proposed method. Experimental rate of convergence is also provided
in this section. We summarize the paper and present concluding remarks
in the final section.

\vskip .25in \section{\label{sec:Pre}\bf{Preliminaries}}

\vskip .1in

\noindent
In this section, we provide some definitions and properties of fractional
derivatives and product integration rules which are used further in
the paper.

\subsection{Factional calculus}

\noindent
There are different definitions for fractional derivative. Since the
Caputo definition allows imposing initial and boundary conditions,
it is preferred in applications, so we focus on this definition in
the paper.
\begin{definition}
\label{CaputoDef}\citep{diethelm} The Caputo fractional derivative
of order $\alpha>0$ of the function $f$ is given by 
\begin{eqnarray}
{D_{x}^{\alpha}}f(x)=\frac{1}{\Gamma(m-\alpha)}\int_{0}^{x}\frac{1}{(x-s)^{\alpha-m+1}}\frac{d^{m}f(s)}{ds^{m}}ds,\quad x>0,\label{eq:CaputoDef}
\end{eqnarray}
where $m=\left\lceil \alpha\right\rceil $. 
\end{definition}
According to Definition \ref{CaputoDef}, the Caputo temporal and
spatial fractional derivative of order $\alpha$ of the function $u(x,t)$
are respectively given by
\begin{eqnarray}
 &  & {D_{t}^{\alpha}}u(x,t)=\frac{1}{\Gamma(m-\alpha)}\int_{0}^{t}\frac{1}{(t-s)^{\alpha-m+1}}\frac{\partial^{m}u(x,s)}{\partial s^{m}}ds,\label{eq:4 CaputoFracDef}\\
 &  & {D_{x}^{\alpha}}u(x,t)=\frac{1}{\Gamma(m-\alpha)}\int_{0}^{x}\frac{1}{(x-s)^{\alpha-m+1}}\frac{\partial^{m}u(s,t)}{\partial s^{m}}ds.
\end{eqnarray}

\subsection{\label{subsec:PI}Product integration}

\noindent
Product integration method, originally proposed by Young \citep{Young},
is of convolution type quadrature. It is usually used for integral
equations with singular kernels \citep{Allo,Weiss}. Recently, this
technique has been used for fractional ODEs \citep{Garr}. Here, we
introduce the basic idea behind this simple and powerful technique
(see \citep{atkinson} for a detailed description).

Consider the integral operator
\begin{eqnarray}
\mathcal{K}f(x)=\int_{a}^{b}H(x,s)L(x,s)f(s)ds,\quad a\leq x\leq b,\label{eq:Volt}
\end{eqnarray}
in which $f$ is a smooth function, $K(x,t)=H(x,t)L(x,t)$ is the
kernel of the integral operator such that $\mathcal{K}$ is a compact
operator, $L$ is a well-behaved function and $H$ is a weakly singular
function on $[a,b]\times[a,b]$. Consider the grid space $\{t_{i}=a+ih:\,i=0,...,n\}$
with grid length $h=\frac{b-a}{n}$. To provide an approximation for
(\ref{eq:Volt}), the product integration's idea is to replace $L(x,s)f(s)$
with an interpolating polynomial based on nodes $t_{i},\,i=0,...,n$,
lets denote by $[L(x,s)f(s)]_{n}$ and the remaining terms in (\ref{eq:Volt})
are then evaluated exactly. To be more precise, we use the common
way of approximating $L(x,s)f(s)$, namely Newton Cotes quadrature
formulas. Consider the piecewise linear interpolation
\begin{eqnarray}
[L(x,s)f(s)]_{n}=\frac{1}{h}\left(\left(s-t_{j}\right)L(x,t_{j+1})f(t_{j+1})-\left(s-t_{j+1}\right)L(x,t_{j})f(t_{j})\right),\label{eq:appr}
\end{eqnarray}
for $t_{j}\leq s\leq t_{j+1},\:j=0,...,n-1,$ and $a\leq x\leq b$.
Putting (\ref{eq:appr}) into (\ref{eq:Volt}) gives the following
approximation
\begin{eqnarray*}
\mathcal{K}_{n}f(x)=\sum_{j=0}^{n}w_{j}(t)L(x,t_{j})f(t_{j})
\end{eqnarray*}
where the quadrature weights are obtained as 
\begin{eqnarray*}
 &  & w_{0}(t)=\frac{1}{h}\int_{t_{0}}^{t_{1}}{\left(t_{1}-s\right)H(x,s)ds},\\
 &  & w_{j}(t)=\frac{1}{h}\left(\int_{t_{j-1}}^{t_{j}}{\left(s-t_{j-1}\right)H(x,s)ds}+\int_{t_{j}}^{t_{j+1}}{\left(t_{j+1}-s\right)H(x,s)ds}\right),\quad1\leq j\leq n-1,\\
 &  & w_{n}(t)=\frac{1}{h}\int_{t_{n-1}}^{t_{n}}{\left(s-t_{n-1}\right)H(x,s)ds}.
\end{eqnarray*}
The function $H$ is selected such that the weights can be evaluated
exactly. In this paper, since we use fractional operators, it is assumed
that $H(x,s)=\frac{1}{\left(x-s\right)^{\eta}}$ for some $0<\eta<1$.

\section{\label{sec:Bern}Bernstein polynomial basis}

\vskip .1in

\noindent
Bernstein polynomials provide a flexible tool in solving differential
and integral equations \citep{Behi,Doha,Jani,Saad} as well as in
computer graphics \citep{Fari}. Bernstein polynomials of degree $N$
are defined on the interval $[a,b]$ as follows:
\begin{eqnarray}
B_{i,N}(x)=\frac{\dbinom{N}{i}(x-a)^{i}(b-x)^{N-i}}{(b-a)^{N}},\qquad0\leq i\leq N.\label{eq:1 BernDef}
\end{eqnarray}
In this paper, we suppose that these polynomials are zero for the
cases $i<0$ and $i>N$. The set $\left\{ B_{i,N}(x):\,i=0,\dots,N\right\} $
forms a basis for $\mathcal{P}_{n}$, the set of polynomials of degree
up to $N$. This basis has many advantages in comparison with other
bases in some applications. Farouki \citep{farouki} showed that the
Bernstein polynomial basis on a given interval is optimally stable
in the sense that no other nonnegative basis yields systematically
smaller condition numbers for the values or roots of arbitrary polynomials
on that interval. Moreover, we have the following properties \citep{Faro}
\begin{eqnarray}
 &  & B_{i,N}(a)=\delta_{i,0},\quad B_{i,N}(b)=\delta_{i,N},\label{eq:BernBound}\\
 &  & 0\leq B_{i,N}(x)\leq1,\:\sum_{i=0}^{N}B_{i,N}(x)=1,\\
 &  & B_{i,N}^{\prime}(x)=\frac{N}{b-a}\left(B_{i-1,N-1}(x)-B_{i,N-1}(x)\right),\label{eq:2.7}\\
 &  & B_{i,N-1}(x)=\frac{1}{N}\left[\left(N-i\right)B_{i,N}(x)+\left(i+1\right)B_{i+1,N}(x)\right],\label{eq:2.8}
\end{eqnarray}
for $0\leq i\leq N$. The symbol $\delta$ is the Kronecker symbol. 

For a better formulation of our method, by combining (\ref{eq:2.8})
and (\ref{eq:2.7}) and some simplifications, we derive the following
result to express derivatives of Bernstein basis of degree $N$ in
terms of the same basis.
\begin{theorem}
\label{LemmaBern}For the first and second order derivatives of Bernstein
polynomials, we have the following three-term and five-terms properties
\begin{eqnarray}
 &  & B_{i,N}^{\prime}(x)=\frac{1}{b-a}\sum_{k=-1}^{1}{d_{k,i}^{(1)}B_{i+k,N}(x)},\label{eq:2.9}\\
 &  & B_{i,N}^{\prime\prime}(x)=\frac{1}{\left(b-a\right)^{2}}\sum_{k=-2}^{2}{d_{k,i}^{(2)}B_{i+k,N}(x)},\label{eq:2.10}
\end{eqnarray}
for $0\leq i\leq N$, where the coefficients are given by
\[
\left\{ \begin{array}{l}
d_{-1,i}^{(1)}=N-i+1,\\
d_{-2,i}^{(2)}=\left(N-i+2\right)\left(N-i+1\right),\\
d_{0,i}^{(2)}=N^{2}-6Ni+6i^{2}-N,\\
d_{2,i}^{(2)}=\left(i+2\right)\left(i+1\right).
\end{array}\begin{array}{l}
d_{0,i}^{(1)}=-\left(N-2i\right),\quad d_{1,i}^{(1)}=-(i+1),\\
d_{-1,i}^{(2)}=-2\left(N-i+1\right)\left(N-2i+1\right),\\
d_{1,i}^{(2)}=2\left(i+1\right)\left(N-2i-1\right),\\
\\
\end{array}\right.
\]
\end{theorem}
\begin{remark}
Corresponding to the non-orthogonal Bernstein polynomial basis, the
associated orthonormalized basis, denoted by $\phi_{i,N}(x)$, is
obtained by using Gram-Schmidt algorithm. Bellucci \citep{bellucci}
derived an explicit formula for $\phi_{i,N}(x)$ on the unit interval
$[0,1]$ as
\begin{eqnarray*}
\phi_{i,N}(x)=\sqrt{2(N-i)+1}(1-x)^{N-i}\sum_{k=0}^{j}(-1)^{k}\binom{2N+1-k}{i-k}\binom{i}{k}x^{i-k}.
\end{eqnarray*}
Although it has many advantages due to orthogonality, this basis does
not have the three- and five-term formula similar to (\ref{eq:2.8})
and (\ref{eq:2.10}). For example, for $i=4$ and $N=4$, we obtain
\begin{eqnarray*}
\phi_{i,N}^{\prime}(x)=\sum_{i=0}^{N}c_{i}\phi_{i,N}(x)
\end{eqnarray*}
with all nonzero coefficient $c_{0},c_{1},c_{2},c_{3},c_{4}.$ It
is worth noting that even Chebyshev and Legendre polynomials also
do not have three- and five-terms relation for derivatives (see Sections
3.3 and 3.4 of \citep{Shen} for these polynomials and derivative
relations). We choose the Bernstein polynomials to formulate our method
in Section \ref{sec:Main}.
\end{remark}
It is a well-known result that for a continuous function $f$ on the
unit interval, the Bernstein approximation polynomial defined by
\begin{eqnarray*}
p_{n}(x)=\sum_{k=0}^{n}{f\left(\frac{k}{n}\right)B_{k,n}\left(x\right)},
\end{eqnarray*}
has the following asymptotic property when $f^{\prime\prime}(x)\neq0$
\citep{Voro} 
\begin{eqnarray*}
\lim_{n\rightarrow\infty}n\left(f\left(x\right)-p_{n}\left(x\right)\right)=\frac{x\left(1-x\right)}{2}f^{\prime\prime}\left(x\right).
\end{eqnarray*}

The next two sections are devoted to providing a numerical method
for the problem (\ref{eq:main})-(\ref{BVs}) using Bernstein polynomials
with collocation method utilizing a product integration technique.

\section{\label{sec:Time-discret}Time discretization}

\vskip .1in

\noindent
In this section, we describe discretization of the time fractional
derivative. The fractional derivative uses function information on
a continuous interval, this discretization only evaluates the function
at some nude points so it leads to less computational complexity.
This scheme is a common way in the numerical methods for time dependent
FPDEs \citep{Deng,Gao,Rame}.

Without loss of generality, we consider the problem (\ref{eq:main})-(\ref{BVs})
on the bounded domain $\Omega=[0,1]\times[0,T]$. Let $u_{k}(x):=u(x,t_{k})$
for $k=0,1,..,M,$ where $t_{k}=k\tau$ and $\tau=\frac{T}{M}$ is
the time step length. By using (\ref{eq:4 CaputoFracDef}), the time
fractional derivative at time $t_{k+1}$ in the left side of the equation
(\ref{eq:main}) is formulated as
\begin{eqnarray*}
{D_{t}^{\alpha}}u(x,t_{k+1}) & = & \frac{1}{\Gamma(1-\alpha)}\sum_{j=0}^{k}\int_{t_{j}}^{t_{j+1}}\frac{\frac{\partial u}{\partial s}(x,s)}{(t_{k+1}-s)^{\alpha}}ds\\
 & = & \frac{1}{\Gamma(1-\alpha)}\sum_{j=0}^{k}\frac{u(x,t_{j+1})-u(x,t_{j})}{\tau}\int_{t_{j}}^{t_{j+1}}\frac{ds}{(t_{k+1}-s)^{\alpha}}+r_{\tau}^{k+1}.
\end{eqnarray*}
from which it follows that
\begin{eqnarray}
{D_{t}^{\alpha}}u(x,t_{k+1})=\mu_{\tau}^{\alpha}\sum_{j=0}^{k}a_{k,j}^{\alpha}\left(u\left(x,t_{j+1}\right)-u\left(x,t_{j}\right)\right)+r_{\tau}^{k+1},
\end{eqnarray}
where $\mu_{\tau}^{\alpha}=\frac{1}{\tau^{\alpha}\Gamma(2-\alpha)},\,a_{kj}^{\alpha}=\left(k+1-j\right)^{1-\alpha}-\left(k-j\right)^{1-\alpha}.$
The error bound for truncation error is given by
\begin{eqnarray}
r_{\tau}^{k+1}\leq\tilde{c}_{u}\tau^{2-\alpha},\label{eq:ErrTimeFrac}
\end{eqnarray}
where the coefficient $\tilde{c}_{u}$ is a constant depending only
on $u$ \citep{Deng}. 

Similar to \citep{Deng,Gao}, we consider the time discrete fractional
differential operator $L_{t}^{\alpha}$ as
\begin{eqnarray}
L_{t}^{\alpha}u(x,t_{k+1})=\mu_{\tau}^{\alpha}\sum_{j=0}^{k}a_{k,j}^{\alpha}\left(u_{j+1}\left(x\right)-u_{j}\left(x\right)\right),\label{eq:Time}
\end{eqnarray}
in which $u_{j}(x):=u(x,t_{j+1})$. So we have
\begin{eqnarray*}
{D_{t}^{\alpha}}u(x,t_{k+1})=L_{t}^{\alpha}u(x,t_{k+1})+r_{\tau}^{k+1}.
\end{eqnarray*}
From (\ref{eq:ErrTimeFrac}), it is seen that the order of convergence
is $O(\tau^{2-\alpha}).$ The discretized operator $L_{t}^{\alpha}$
is called the L1 approximation operator \citep{Gao}.

\section{\label{sec:Main}Formulation and algorithm for space-time fractional
advection-dispersion equation (\ref{eq:main})}

\vskip .1in

\noindent
We rewrite the advection-dispersion equation (\ref{eq:main}) at horizontal
time line $t=t_{k+1}$ as
\begin{eqnarray*}
{D_{t}^{\alpha}}u(x,t_{k+1})=\kappa_{1}{D_{x}^{\beta}}u(x,t_{k+1})-\kappa_{2}D_{x}^{\gamma}u(x,t_{k+1}).
\end{eqnarray*}
Using $L_{t}^{\alpha}u(x,t_{k+1})$ given by (\ref{eq:Time}) as an
approximation for ${D_{t}^{\alpha}}u(x,t_{k+1})$, it yields the following
time-discrete scheme for (\ref{eq:main})
\begin{eqnarray}
\mu_{\tau}^{\alpha}\sum_{j=0}^{k}a_{k,j}^{\alpha}\left(u_{j+1}\left(x\right)-u_{j}\left(x\right)\right)=\kappa_{1}{D_{x}^{\beta}}u_{k+1}(x)-\kappa_{2}{D_{x}^{\gamma}}u_{k+1}(x),\label{eq1}
\end{eqnarray}
for $k=0,...,M-1.$ In the following, we discuss about approximating
the terms involving spatial fractional derivatives. 

Let $N$ be a positive integer and $h=\frac{1}{N}$ be the grid size
in the $x$-direction and $x_{r}=\frac{r}{h},\:r=0,...,N.$ We have
the following theorem for a simple evaluation of spatial fractional
derivatives at collocation points.
\begin{theorem}
\label{LemmaAdv} Suppose that $\partial_{x}^{4}u(x,t)$ is continuous
on $\Omega$, $0\leq k\leq M-1$, $1<\beta<2$. Then for $r=1,\ldots.,N-1$,
we have
\begin{eqnarray}
{D_{x}^{\beta}}u_{k+1}(x_{r})=\nu_{h}^{\beta}\sum_{j=0}^{r}{w_{j,r}^{\beta}u_{k+1}^{\prime\prime}(x_{j})}+R_{h}^{\beta}(r),\label{eq:Advection}
\end{eqnarray}
where $\nu_{h}^{\beta}=\frac{h^{2-\beta}}{\Gamma(2-\beta)}$. The
weighting coefficients are given by
\[
\begin{array}{ll}
w_{j,r}^{\beta}=\tilde{w}_{j-1,0}^{\beta}-\tilde{w}_{j,1}^{\beta}, & j=0,...,r,\\
\tilde{w}_{r-j,\rho}^{\beta}=c_{j}^{3-\beta}-\left(j-\rho\right)c_{j}^{2-\beta}, & \rho=0,1,\,j=0,\ldots.,r,\;(j,\rho)\neq(0,1),
\end{array}
\]
with $\tilde{w}_{-1,0}^{\beta}=0$ and $\tilde{w}_{r,1}^{\beta}=0$.
The error is bounded as
\begin{eqnarray}
\left\Vert R_{h}^{\beta}\right\Vert _{\infty}\leq\frac{h^{2}\mathcal{M}_{u}}{8\Gamma(3-\beta)},\label{eq:ErrAdv}
\end{eqnarray}
where $\mathcal{M}_{u}$ depends only on $u$.
\end{theorem}
\begin{proof}
Based on Definition \ref{CaputoDef}, and utilizing product integration
method as explained in Section \ref{subsec:PI}, we can write
\begin{eqnarray*}
 &  & {D_{x}^{\beta}}u_{k+1}(x_{r})=\frac{1}{\Gamma(2-\beta)}\int_{0}^{x_{r}}\frac{1}{(x_{r}-s)^{\beta-1}}\frac{d^{2}u_{k+1}(s)}{ds^{2}}ds\\
 &  & \approx\frac{1}{\Gamma(2-\beta)}\int_{0}^{x_{r}}\frac{1}{(x_{r}-s)^{\beta-1}}\left[\frac{d^{2}u_{k+1}(s)}{ds^{2}}\right]_{r}ds\\
 &  & =\frac{1}{\Gamma(2-\beta)}\sum_{j=0}^{r-1}\int_{x_{j}}^{x_{j+1}}\frac{1}{(x_{r}-s)^{\beta-1}}\frac{\left(s-x_{j}\right)u_{k+1}^{\prime\prime}(x_{j+1})-\left(s-x_{j+1}\right)u_{k+1}^{\prime\prime}(x_{j})}{h}ds\\
 &  & =\frac{h^{2-\beta}}{\Gamma(2-\beta)}\left(-\left(c_{r}^{3-\beta}-\left(r-1\right)c_{r}^{2-\beta}\right)u_{k+1}^{\prime\prime}(x_{0})\right.\\
 &  & +\sum_{j=1}^{r-1}\left(\left(c_{r-j+1}^{3-\beta}-\left(r-j+1\right)c_{r-j+1}^{2-\beta}\right)-\left(c_{r-j}^{3-\beta}-\left(r-j-1\right)c_{r-j}^{2-\beta}\right)\right)u_{k+1}^{\prime\prime}(x_{j})\\
 &  & \left.+\left(\left(c_{1}^{3-\beta}-c_{1}^{2-\beta}\right)\right)u_{k+1}^{\prime\prime}(x_{r})\right),
\end{eqnarray*}
where $c_{j}^{\eta}:=\frac{(j-1)^{\eta}-j^{\eta}}{\eta}$ for $\eta=2-\beta$
and $\eta=3-\beta$. This is simply written as (\ref{eq:Advection})
by applying the notations in the theorem. Using Lagrange interpolation
error formula, we have
\begin{eqnarray*}
\left|R_{h}^{\beta}(r)\right| & = & \left|\frac{1}{\Gamma(2-\beta)}\sum_{j=0}^{r-1}{\int_{x_{j}}^{x_{j+1}}\frac{(s-x_{j})(s-x_{j+1})u_{k+1}^{(4)}(\xi_{j}(s))}{2(x_{r}-s)^{\beta-1}}ds}\right|\\
 & = & \frac{h^{4-\beta}}{\Gamma(2-\beta)}\sum_{j=0}^{r-1}{\left|u_{k+1}^{(4)}(\bar{\xi}_{j})\int_{r-j-1}^{r-j}{\frac{(r-j-t)(r-j-1-t)}{2t^{\beta-1}}dt}\right|}\\
 & \leq & \frac{h^{4-\beta}}{8\Gamma(3-\beta)}\left\Vert u_{k+1}^{(4)}\right\Vert _{\infty}\sum_{j=0}^{r-1}{(j+1)^{2-\beta}-j^{2-\beta}}\\
 & \leq & \frac{h^{4-\beta}r^{2-\beta}\mathcal{M}_{u}}{8\Gamma(3-\beta)}\\
 & \leq & \frac{h^{2}\mathcal{M}_{u}}{8\Gamma(3-\beta)}\left(\frac{N-1}{N}\right)^{2-\beta}\\
 & \leq & \frac{h^{2}\mathcal{M}_{u}}{8\Gamma(3-\beta)},
\end{eqnarray*}
for $r=1,2,\ldots.,N-1$ and the continuity of $\frac{\partial^{4}u}{\partial x^{4}}(x,t)$
on the closed square $\Omega$ implies that it is bounded, i.e., $\mathcal{M}_{u}=\sup_{(x,t)\in\Omega}{|\frac{\partial^{4}u}{\partial x^{4}}(x,t)|}<\infty$.
Due to the continuity of $\frac{\partial^{4}u}{\partial x^{4}}(x,t)$
and noting that $(s-x_{j})(s-x_{j+1})/(x_{r}-s)^{\beta-1}$ does not
change sign in the interval of integration, we used the first mean
value theorem for integrals in the second equality. This concludes
the proof.
\end{proof}

Next we present the following theorem which can be proved analogous
to Theorem \ref{LemmaAdv}. So we omit the proof. 
\begin{theorem}
\label{LemmaDisp}Suppose that $\partial_{x}^{3}u(x,t)$ is continuous
on $\Omega$, $0\leq k\leq M-1$, $1<\beta<2$. Then for $r=1,\ldots.,N-1$,
we have 
\begin{eqnarray}
{D_{x}^{\gamma}}u_{k+1}(x_{r})=\nu_{h}^{\gamma}\sum_{j=0}^{r}{w_{j,r}^{\gamma}u_{k+1}^{\prime}(x_{j})}+R_{h}^{\gamma}(r),\label{eq:Dispersion}
\end{eqnarray}
where $\nu_{h}^{\gamma}=\frac{h^{1-\gamma}}{\Gamma(1-\gamma)}$ and
the weighting coefficients are given by
\[
\begin{array}{lc}
w_{j,r}^{\gamma}=\tilde{w}_{j-1,0}^{\gamma}-\tilde{w}_{j,1}^{\gamma}, & j=0,...,r,\\
\tilde{w}_{r-j,\rho}^{\gamma}=c_{j}^{2-\gamma}-\left(j-\rho\right)c_{j}^{1-\gamma}, & \rho=0,1,\,j=0,\ldots.,r,\;(j,\rho)\neq(0,1),
\end{array}
\]
and as before we set $\tilde{w}_{-1,0}^{\beta}=0$ and $\tilde{w}_{r,1}^{\beta}=0$
and the error is bounded as
\begin{eqnarray}
\left\Vert R_{h}^{\gamma}\right\Vert _{\infty}\leq\frac{h^{2}\mathcal{\bar{M}}_{u}}{8\Gamma(2-\gamma)}.\label{eq:ErrDisp}
\end{eqnarray}
where $\mathcal{\bar{M}}_{u}=\sup_{(x,t)\in\Omega}{|\frac{\partial^{3}u}{\partial x^{3}}(x,t)|}<\infty.$.
\end{theorem}
Using Theorems \ref{LemmaAdv} and \ref{LemmaDisp}, we define the
discrete spatial fractional operators $L_{x}^{\beta}u$ and $L_{x}^{\gamma}u$
by
\begin{eqnarray*}
L_{x}^{\beta}u(x_{r},t_{k+1})=\nu_{h}^{\beta}\sum_{j=0}^{r}{w_{j,r}^{\beta}u_{k+1}^{\prime\prime}(x_{j})},\quad L_{x}^{\gamma}u(x_{r},t_{k+1})=\nu_{h}^{\gamma}\sum_{j=0}^{r}{w_{j,r}^{\gamma}u_{k+1}^{\prime}(x_{j})},
\end{eqnarray*}
as approximations for $D_{x}^{\beta}u$ and $D_{x}^{\gamma}u$. Plugging
these approximations into equation (\ref{eq1}) at nodes $x_{r},\,r=1,...,N-1$,
we obtain the following recursive relation
\begin{eqnarray*}
\mu_{\tau}^{\alpha}\sum_{j=0}^{k}{a_{k,j}^{\alpha}\left(u_{j+1}\left(x_{r}\right)-u_{j}\left(x_{r}\right)\right)}=\kappa_{1}\nu_{h}^{\beta}\sum_{j=0}^{r}{w_{j,r}^{\beta}u_{k+1}^{\prime\prime}(x_{j})}-\kappa_{2}\nu_{h}^{\gamma}\sum_{j=0}^{r}{w_{j,r}^{\gamma}u_{k+1}^{\prime}(x_{j})},
\end{eqnarray*}
for $k\geq0$, from which, with simplification, we obtain
\begin{eqnarray}
\mu_{\tau}^{\alpha}u_{k+1}(x_{r})-\kappa_{1}\nu_{h}^{\beta}\sum_{j=0}^{r}{w_{j,r}^{\beta}u_{k+1}^{\prime\prime}(x_{j})}+\kappa_{2}\nu_{h}^{\gamma}\sum_{j=0}^{r}{w_{j,r}^{\gamma}u_{k+1}^{\prime}(x_{j})}=f_{k+1}(x_{r}),\label{VIDE-1}
\end{eqnarray}
where $f_{k+1}(x)=\mu_{\tau}^{\alpha}\left(u_{k}(x)-\sum_{j=0}^{k-1}a_{k,j}^{\alpha}\left(u_{j+1}\left(x\right)-u_{j}\left(x\right)\right)\right)$. 

From the error formulas (\ref{eq:ErrTimeFrac}), (\ref{eq:ErrAdv})
and (\ref{eq:ErrDisp}), the perturbation error for the solution of
problem (\ref{eq:main}) on the domain $\Omega$, when applying spatial
fractional approximations (\ref{eq:Advection}), (\ref{eq:Dispersion})
in conjunction with time fractional approximation (\ref{eq:Time}),
is obtained as
\begin{eqnarray}
\max_{k\geq0}{\left\Vert u(x,t_{k+1})-u_{k+1}(x)\right\Vert _{\infty}}\leq\tilde{c}_{u}\tau^{2-\alpha}+\frac{h^{2}}{8}\left(\frac{\kappa_{1}}{\Gamma(3-\beta)}+\frac{\kappa_{2}}{\Gamma(2-\gamma)}\right)\mathcal{M},\label{eq:TotalError}
\end{eqnarray}
where $\mathcal{M}=max(\mathcal{M}_{u},\bar{\mathcal{M}_{u}})$ depends
only on $u$. This shows the convergence of the proposed method and
also it is seen that the method has order of convergence two for space
and $2-\alpha$ for time. 

In order to obtain the approximate solution of the problem (\ref{eq:main})-(\ref{BVs})
, we solve (\ref{VIDE-1}) using Bernstein polynomial basis. Consider
\begin{eqnarray*}
u_{k+1}(x)\approx\sum_{i=0}^{N}{c_{i,k+1}B_{i,N}(x)},\quad k\geq0
\end{eqnarray*}
as an approximate solution at time step $t_{k+1}$. Using the boundary
conditions (\ref{BVs}), we have $u_{k+1}(0)=0$ and $u_{k+1}(1)=0$.
Now from (\ref{eq:BernBound}), we obtain $c_{0,k+1}=0$, $c_{N,k+1}=0$
and so
\begin{eqnarray}
u_{k+1}(x)\approx\sum_{i=1}^{N-1}{c_{i,k+1}B_{i,N}(x)}.\label{eq:BaseFunc}
\end{eqnarray}
Substituting (\ref{eq:BaseFunc}) accompanied with (\ref{eq:2.9})
and (\ref{eq:2.10}) into (\ref{VIDE-1}), we obtain the following
linear system of equations
\begin{eqnarray}
 &  & \sum_{i=1}^{N-1}{\mathcal{A}_{r,i}c_{i,k+1}}=f_{k+1}(x_{r}),\,1\leq r<N,\label{mainSystem}
\end{eqnarray}
at each time level $k=0,1,\ldots.,M-1,$ where the coefficients are
given by
\begin{eqnarray*}
\mathcal{A}_{r,i} & = & \mu_{\tau}^{\alpha}B_{i,N}(x_{r})-\sum_{j=0}^{r}\left\{ \kappa_{1}\nu_{h}^{\beta}w_{j,r}^{\beta}\sum_{s=-2}^{2}{d_{s,i}^{(2)}B_{i+s,N}(x_{j})}\right..\\
 &  & \left.\hfill\qquad+\kappa_{2}\nu_{h}^{\gamma}w_{j,r}^{\gamma}\sum_{s=-1}^{1}{d_{s,i}^{(1)}B_{i+s,N}(x_{j})}\right\} 
\end{eqnarray*}

Now we write (\ref{mainSystem}) in the matrix form as
\begin{eqnarray*}
\left(\mu_{\tau}^{\alpha}\mathcal{B}-\kappa_{1}\nu_{h}^{\beta}\mathcal{W}_{\beta}\mathcal{D}_{2}+\kappa_{2}\nu_{h}^{\gamma}\mathcal{W}_{\gamma}\mathcal{D}_{1}\right)\boldsymbol{c}_{k+1}=\boldsymbol{f}_{k+1},\quad k\geq0
\end{eqnarray*}
where $\mathbf{c}_{k+1}^{T}=[c_{1,k+1},...,c_{N-1,k+1}]$, $\boldsymbol{f}_{k+1}^{T}=[f_{1,k+1},...,f_{N-1,k+1}]$
and $f_{r,k+1}:=f_{k+1}(x_{r}),\:1\leq r<N$. Let $\boldsymbol{\phi}^{T}(x)=[B_{1,N}(x),...,B_{N-1,N}(x)]$
be the vector of basis. The solution (\ref{eq:BaseFunc}) is written
as $u_{k+1}(x)=\mathbf{c}_{k+1}^{T}\boldsymbol{\phi}(x)$ and the
square matrices $\mathcal{B}$, $\mathcal{W}_{\beta}$ and $\mathcal{W}_{\gamma}$
are defined as
\begin{eqnarray}
 &  & \mathcal{B}=\left[\begin{array}{cccc}
B_{1,N}(x_{1}) & B_{2,N}(x_{1}) & \cdots & B_{N-1,N}(x_{1})\\
B_{1,N}(x_{2}) & B_{2,N}(x_{2}) & \cdots & B_{N-1,N}(x_{2})\\
\vdots & \vdots & \ddots & \vdots\\
B_{1,N}(x_{N-1}) & B_{2,N}(x_{N-1}) & \cdots & B_{N-1,N}(x_{N-1})
\end{array}\right],\\
 &  & \mathcal{W}_{\eta}=\left[\begin{array}{ccccc}
w_{0,1}^{\eta} & w_{1,1}^{\eta} & 0 & \cdots & 0\\
w_{0,2}^{\eta} & w_{1,2}^{\eta} & w_{2,2}^{\eta} & \cdots & 0\\
\vdots & \vdots & \vdots & \ddots & \vdots\\
w_{0,N-2}^{\eta} & w_{1,N-2}^{\eta} & w_{2,N-2}^{\eta} & \cdots & w_{N-2,N-2}^{\eta}\\
\ w_{0,N-1}^{\eta} & w_{1,N-1}^{\eta} & w_{2,N-1}^{\eta} & \cdots & w_{N-2,N-1}^{\eta}
\end{array}\right],
\end{eqnarray}
for $\eta=\beta,\,\gamma$. The matrix $\mathcal{B}$ is well known
as Bernstein collocation matrix that is a totally positive matrix,
i.e., all its minors are nonnegative \citep{Delgado}, and the matrices
$\mathcal{W_{\beta}}$ and $\mathcal{W}_{\eta}$ are lower Hessenberg
matrices. Also $\mathcal{D}_{1}$ and $\mathcal{D}_{2}$ are as follows:
\[
\mathcal{D}_{p}=\sum_{s=-p}^{p}d_{s,i}^{(p)}\left[\begin{array}{cccc}
B_{1+s,N}(x_{0}) & B_{2+s,N}(x_{0}) & \cdots & B_{N-1+s,N}(x_{0})\\
B_{1+s,N}(x_{1}) & B_{2+s,N}(x_{1}) & \cdots & B_{N-1+s,N}(x_{1})\\
\vdots & \vdots & \ddots & \vdots\\
B_{1+s,N}(x_{N-1}) & B_{2+s,N}(x_{N-1}) & . & B_{N-1+s,N}(x_{N-1})
\end{array}\right],
\]
where $p=1,2$. Note that each matrix in this summation can be written
as a permutation of the Bernstein collocation matrix.

Note that for the initial time step, the solution is given by the
initial condition (\ref{IV}) as $u_{0}(x)=g(x)$ and for the remaining
steps, the linear system (\ref{mainSystem}) is solved to obtain the
vecot $\mathbf{c}_{k+1}$ and then the solution is obtained as (\ref{eq:BaseFunc})
at time step $k+1$. The procedure is repeated for $k=0,1,\dots$
until the desired time step is reached. 

\section{\label{sec:Ex}Numerical experiments and discussion}

\vskip .1in

\noindent
Due to nonlocal feature of fractional derivatives, even local-based
numerical techniques like finite element method applied to space-time
fractional problems do not lead to sparse linear systems \citep{Deng}.
But for the cases that only one of time or space derivatives is fractional,
it is possible to develop some methods that lead to a system of equations
with sparse coefficient matrix. For instance, authors of \citep{Gao}
obtained a triple-tridiagonal structure for time-fractional advection-dispersion
equation. The matrix formulation of the method proposed in this paper
shows that the resultant matrix is not sparse but the following numerical
experiments show that acceptable results are obtained with relatively
small resultant matrices.

\begin{table}
\centering{}\caption{\label{tab:varying beta}$L_{2}$ and $L_{\infty}$ error norms for
the scheme (\ref{mainSystem}) for different values of $\beta$ }
\begin{tabular}{cccccc}
\hline 
$\beta$ & $h$ & $E_{2}^{T}$ & $Rate_{2}$ & $E_{\infty}^{T}$ & $Rate_{\infty}$\tabularnewline
\hline 
\multirow{3}{*}{1.2} & 1/4 & 3.685109E-03 &  & 4.927313E-03 & \tabularnewline
 & 1/8 & 9.198267E-04 & 2.0022 & 1.257808E-03 & 1.9698\tabularnewline
 & 1/16 & 2.380429E-04 & 1.9501 & 3.278200E-04 & 1.9399\tabularnewline
\hline 
\multirow{3}{*}{1.4} & 1/4 & 3.684945E-03 &  & 4.927401E-03 & \tabularnewline
 & 1/8 & 9.198071E-04 & 2.0022 & 1.257786E-03 & 1.9699\tabularnewline
 & 1/16 & 2.314130E-04 & 1.9908 & 3.211054E-04 & 1.9697\tabularnewline
\hline 
\multirow{3}{*}{1.6} & 1/4 & 3.683697E-03 &  & 4.924700E-03 & \tabularnewline
 & 1/8 & 9.195621E-04 & 2.0021 & 1.257504E-03 & 1.9694\tabularnewline
 & 1/16 & 2.168473E-04 & 2.0842 & 3.028411E-04 & 2.0539\tabularnewline
\hline 
\multirow{3}{*}{1.8} & 1/4 & 3.681100E-03 &  & 4.924700E-03 & \tabularnewline
 & 1/8 & 9.187786E-04  & 2.0023 & 1.256610E-03 & 1.9704\tabularnewline
 & 1/16 & 2.121292E-04 & 2.1147 & 2.994302E-04 & 2.0692\tabularnewline
\hline 
\end{tabular}
\end{table}

We provide some numerical tests to show the efficiency of the proposed
method and to examine the theoretical results of the paper. Without
loss of generality, we add a forcing term $h(x,t)$ to the space-time
fractional advection-dispersion equation (\ref{eq:main}). 

As the first example, consider the problem (\ref{eq:main}) with initial
condition $g(x)=x^{2}(1-x)^{2}$, homogeneous boundary conditions,
the time and space fractional orders $\alpha=0.4$, $\gamma=0.5$
and anomalous dispersion and advection coefficients $\kappa_{1}=0.001,\,\kappa_{2}=2$,
respectively with the exact solution $u=x^{2}(1-x)^{2}exp(-t)$. Table
\ref{tab:varying beta} shows $E_{2}^{T}=\left\Vert u(x,T)-u_{h}^{T}(x)\right\Vert _{2}$
and $E_{\infty}^{T}=\left\Vert u(x,T)-u_{h}^{T}(x)\right\Vert _{\infty}$
errors at time $T=1$ with time step size $\tau=0.05$ for different
values of $\beta\in(1,2)$ and with different space mesh sizes. Also
experimental rate of convergence is provided in this table. It is
seen that the rate of convergence is about two as it is expected from
theoretical results. Tables \ref{tab:varyingAlpha} and \ref{tab:varyingGamma}
show the computational errors and experimental rate of convergence
for the same parameters as \ref{tab:varying beta} but for $\beta=1.5$
and varying $\alpha\in(0,1)$ and also $\alpha=0.5,\,\beta=1.5$ and
different values of $\gamma\in(0,1)$, respectively. 

\begin{table}
\caption{\label{tab:varyingAlpha}Error norms and experimental order of convergence
for different values of $\alpha$ at $T=1$.}

\centering{}%
\begin{tabular}{cccccc}
\hline 
$\alpha$ & $h$ & $E_{2}^{T}$ & $Rate_{2}$ & $E_{\infty}^{T}$ & $Rate_{\infty}$\tabularnewline
\hline 
\multirow{3}{*}{0.2} &  1/4  & 3.315112E-03 &  & 4.475943E-03 & \tabularnewline
 &  1/8  & 8.253774E-04 & 2.0059 & 1.134993E-03 & 1.9795\tabularnewline
 &  1/16  & 2.198645E-04 & 1.9084 & 2.997409E-04 & 1.9209\tabularnewline
\hline 
\multirow{3}{*}{0.4} &  1/4  & 3.684620E-03 &  & 4.927243E-03 & \tabularnewline
 &  1/8  & 9.198521E-04 & 2.0020 & 1.257693E-03 & 1.9700\tabularnewline
 &  1/16  & 2.661032E-04 & 1.7894 & 3.548821E-04 & 1.8254\tabularnewline
\hline 
\multirow{3}{*}{0.6} &  1/4  & 4.141701E-03 &  & 5.482632E-03 & \tabularnewline
 &  1/8  & 1.042701E-03 & 1.9899 & 1.415220E-03 & 1.9538\tabularnewline
 &  1/16  & 3.255528E-04 & 1.6794 & 4.256264E-04 & 1.7334\tabularnewline
\hline 
\multirow{3}{*}{0.8} &  1/4  & 4.769215E-03 &  & 6.246503E-03 & \tabularnewline
 &  1/8  & 1.225821E-03 & 1.9600 & 1.646025E-03 & 1.9241\tabularnewline
 &  1/16  & 3.157964E-04 & 1.9567 & 4.427624E-04 & 1.8944\tabularnewline
\hline 
\end{tabular}
\end{table}

Note that both in formulation of the method and in numerical tests,
for convenience of formulation, the collocation points are assumed
to be equidistant on the unit interval. However, it can be done using
any arbitrary set of collocation points.

\begin{table}
\caption{\label{tab:varyingGamma}Error norms and experimental order of convergence
for different values of $\gamma$ at $T=1$.}

\centering{}%
\begin{tabular}{cccccc}
\hline 
$\gamma$ & $h$ & $E_{2}^{T}$ & $Rate_{2}$ & $E_{\infty}^{T}$ & $Rate_{\infty}$\tabularnewline
\hline 
\multirow{3}{*}{0.2} &  1/4  & 3.9274E-03 &  & 5.3811E-03 & \tabularnewline
 &  1/8  & 9.8523E-04 & 1.9950 & 1.3545E-03 & 1.9901\tabularnewline
 &  1/16  & 2.5973E-04 & 1.9234 & 3.6105E-04 & 1.9075\tabularnewline
\multirow{3}{*}{0.4} &  1/4  & 3.9810E-03 &  & 5.3830E-03 & \tabularnewline
 &  1/8  & 9.9690E-04 & 1.9976 & 1.3653E-03 & 1.9792\tabularnewline
 &  1/16  & 2.6503E-04 & 1.9113 & 3.6183E-04 & 1.9158\tabularnewline
\multirow{3}{*}{0.6} &  1/4  & 3.7012E-03 &  & 4.8086E-03 & \tabularnewline
 &  1/8  & 9.2097E-04 & 2.0068 & 1.2447E-03 & 1.9498\tabularnewline
 &  1/16  & 2.4507E-04 & 1.9099 & 3.2997E-04 & 1.9154\tabularnewline
\multirow{3}{*}{0.8} &  1/4  & 2.8985E-03 &  & 3.6448E-03 & \tabularnewline
 &  1/8  & 6.6410E-04 & 2.1259 & 8.8712E-04 & 2.0387\tabularnewline
 &  1/16  & 1.7361E-04 & 1.9355 & 2.3733E-04 & 1.9022\tabularnewline
\hline 
\end{tabular}
\end{table}

As the second example, consider the advection dispersion equation
(\ref{eq:main}) with the coefficients $\kappa_{1}=0.1,\,\kappa_{2}=5$
and fractional orders as $\alpha=0.5,\,\beta=1.5$ and $\gamma=0.5$
with initial condition $u(x,0)=0$ and the forcing term, $h(x,s)$
such that the exact solution is $u=\sin{\left(\pi x\right)}t^{2}.$
Table \ref{tab:ErrorsTimeSpace} provides the $L_{2}$ error norms
with overall CPU time (seconds) consumed in the algorithm for obtaining
the numerical solution at $T=1$ .

\begin{table}
\caption{\label{tab:ErrorsTimeSpace}$L_{2}$ norm error and CPU time after
$M=\frac{1}{\tau}$ steps }

\centering{}%
\begin{tabular}{|c|c|c|c|}
\hline 
$h$ & $\tau$ & $E_{2}^{T}$ & CPU time (s)\tabularnewline
\hline 
\hline 
1/4 & 1/10 & 3.4898E-02 & 0.16\tabularnewline
\hline 
1/6 & 1/20 & 1.3529E-02 & 0.843\tabularnewline
\hline 
1/8 & 1/30 & 7.5908E-03 & 2.359\tabularnewline
\hline 
1/10 & 1/40 & 4.8497E-03 & 6.266\tabularnewline
\hline 
\end{tabular}
\end{table}

\section{\label{sec:Con}Concluding remarks}

\vskip .1in

\noindent
In this paper, we proposed a numerical approach for solving apace
and time fractional advection dispersion equations. After time discretization,
we used product integration technique to derive some explicit formulas
for terms involving space fractional derivatives. By utilizing Bernstein
polynomial basis, we transformed the problem into solving a linear
system of algebraic equations. It is worth to mention that the proposed
method is different from the well-known method of line in the sense
that in our method, the problem at first is discretized in time leading
to a boundary value problem instead of an initial value problem in
order to utilize collocation method with Bernstein basis. We also
discussed the error estimation of the approximating relations for
fractional derivatives and using some numerical experiments, we showed
that the method is efficient and simple to implement for solving fractional
advection dispersion equations on a bounded domain. From these numerical
tests, it is seen that the associated experimental order of convergence
is consistent with the analysis. The proposed method can be applied
to a wide range of space and time fractional partial differential
equations on bounded domains.

{\small

}

\end{document}